\documentclass[12pt,reqno]{amsart}
\usepackage{amsfonts}
\usepackage{amsmath}
\usepackage{amsthm}
\usepackage{mathrsfs}
\usepackage{amssymb}
\usepackage{stmaryrd}
\usepackage{graphicx}
\usepackage{epstopdf}
\usepackage{textcomp}

\newtheorem{theorem}{Theorem}
\newtheorem{lemma}{Lemma}

\theoremstyle{definition}

\newtheorem{assumption}{Assumption}

\begin{document}
\title {Rational taxation in an open access fishery model}

\author{Dmitry B. Rokhlin}
\author{Anatoly Usov}

\address{Institute of Mathematics, Mechanics and Computer Sciences,
              Southern Federal University,
Mil'chakova str., 8a, 344090, Rostov-on-Don, Russia}
\email[Dmitry B. Rokhlin]{rokhlin@math.rsu.ru}
\email[Anatoly Usov]{usov@math.rsu.ru}

\thanks{The research is supported by Southern Federal University, project 213.01-07-2014/07.}

\begin{abstract}
We consider a model of fishery management, where $n$ agents exploit a single population with strictly concave continuously differentiable growth function of Verhulst type. If the agent actions are coordinated and directed towards the maximization of
the discounted cooperative revenue, then the biomass stabilizes at the level, defined by the well known ``golden rule''. We show that for independent myopic harvesting agents such optimal (or $\varepsilon$-optimal) cooperative behavior can be stimulated by the proportional tax, depending on the resource stock, and equal to the marginal value function of the cooperative problem. To implement this taxation scheme we prove that the mentioned value function is strictly concave and continuously differentiable, although the instantaneous individual revenues may be neither concave nor differentiable.
\end{abstract}

\maketitle

\section{Introduction}
\label{sec:1}
An unregulated open access to marine resources, where many individual users are involved in the fishery, may easily lead to the over-exploitation or even extinction of fish populations. Moreover, it results in zero rent. These negative consequences of the  unregulated open access (the ''tragedy of commons'': \cite{Har68}) were widely discussed in the literature: see \cite{Gor54,Cla79,Cla06,Arn09}. May be the most evident reason for the occurrence of these phenomena is the myopic behavior of competing harvesting agents, who are interested in the maximization of instantaneous profit flows, and not in the conservation of the population in the long run. In the present paper we consider the problem of rational regulation of an open access fishery, using taxes as the only economical instrument. Other known instruments include fishing quotas of different nature, total allowable catch, limited entry, sole ownership, community rights, various economic restrictions, etc: see, e.g, \cite{Cla06,Arn09}.

Assume for a moment that $n$ agents coordinate their efforts to maximize the aggregated long-run discounted profit. The related aggregated agent, which can be considered as a sole owner of marine fishery resources, conserves the resource under optimal strategy, unless the discounting rate is very large. How such an acceptable cooperative behavior can be realized in practice?

We consider the following scheme. Suppose that some regulator (e.g., the coastal states), being aware of the revenue function and maximal productivity of each agent, declares the amount of proportional tax on catch. Roughly speaking, it turns out that if this tax is equal to the marginal indirect utility (marginal value function) of the cooperative optimization problem, then the myopic profit maximizing agents will follow an optimal cooperative strategy, maximizing the aggregated long-run discounted profit. The idea of using such taxes in harvesting management was often expressed in the bioeconomic literature: see \cite{Cla80}, \cite{McKel89}, \cite[Chapter 10]{HanShoWhi97}, \cite[Chapter 7]{Il09}. Our goal is to study this idea more closely from the mathematical point of view.

The first theoretical question we encounter, trying to implement the mentioned taxation scheme, concerns the differentiability of the value function $v$ of the cooperative problem. Assuming that the population growth function is strictly concave and continuously differentiable, in Sections \ref{sec:2} and \ref{sec:3} we prove $v$ inherits these properties, although the instantaneous revenue functions may be non-concave.

The differentiability of $v$ is proved by the tools from optimal control  and convex analysis. Our approach relies on the characterization of $v$ as the unique solution of the related Hamilton-Jacobi-Bellamn equation.
We neither use the general results like \cite{RinZapSan12}, nor the related technique. At the same time, our results are not covered by \cite{RinZapSan12}. Simultaneously we construct optimal strategies and prove that optimal trajectories are attracted to the biomass level $\widehat x$, defined by the well known ``golden rule''. This level depends on the discounting rate, which is at regulator's disposal.

If the agent revenue function are non-concave, then an optimal solution of the infinite horizon cooperative problem may exist only in the class of relaxed (or randomized) harvesting  strategies. Such strategies can hardly be realized in practice, and certainly cannot be stimulated by taxes. Nevertheless, in Section \ref{sec:4} we show that piecewise constant strategies (known as the ``pulse fishing'') of myopic agents, stimulated by the proportional tax $v'\alpha$ on the fishing intensity $\alpha$, are $\varepsilon$-optimal for the cooperative problem. Moreover, the related trajectory is retained in any desired neighbourhood of $\widehat x$ for large values of time. Finally, we introduce the notion of the critical tax $v'(\widehat x)$ and prove that it can only increase, when the agent community widens.

\section{Cooperative harvesting problem: the case of concave revenues}
\label{sec:2}
Let a population biomass $X$ satisfy the differential equation
\begin{equation} \label{2.1}
X_t=x+\int_0^t b(X_s)\,ds-\sum_{i=1}^n \int_0^t \alpha_s^i\,ds,
\end{equation}
where $b$ is the growth rate of the population, and $\alpha^i$ is the harvesting rate of $i$-th agent. We assume that $b$ is a \emph{differentiable strictly concave} function defined on an open neighbourhood of $[0,1]$, and
$$b(x)>0,\quad x\in (0,1),\quad b(0)=b(1)=0.$$
The widely used Verhulst growth function $b(x)=x(1-x)$ is a typical example. Agent harvesting strategies $\alpha^i$ are (Borel) measurable functions with values in the intervals $[0,\overline \alpha^i]$, $\overline\alpha^i>0$. A harvesting strategy $\alpha=(\alpha^1,\dots,\alpha^n)$ is called \emph{admissible} if the solution $X^{x,\alpha}$ of (\ref{2.1}) stays in $[0,1]$ forever: $X_t^{x,\alpha}\in [0,1]$, $t\ge 0$. Note that for given $\alpha$ the solution $X^{x,\alpha}$ is unique, since $b$, being concave, is Lipschitz continuous. The set of admissible strategies, corresponding to an initial condition $x$, is denoted by $\mathscr A_n(x)$.

Consider the cooperative objective functional
$$ J_n(x,\alpha)=\sum_{i=1}^n\int_0^\infty e^{-\beta t} f_i(\alpha_t^i)\,dt,\quad \beta>0$$
of agent community. \emph{We always assume} that the instantaneous revenue function $f_i:[0,\overline \alpha^i]\mapsto\mathbb R_+$ of $i$-th agent is at least \emph{continuous}, and $f_i(0)=0$. Let
\begin{equation} \label{2.2}
 v(x)=\sup_{\alpha\in\mathscr A_n(x)} J_n(x,\alpha),\quad x\in [0,1]
\end{equation}
be the value function of the cooperative optimization problem.

When studying the properties of the value function it is convenient to reduce the dimension of the control vector to $1$. Recall that the function
$$ (g_1\oplus\dots\oplus g_n)(x)=\sup\{g_1(x_1)+\dots+g_n(x_n):x_1+\dots+x_n=x\}$$
is called the \emph{infimal convolution} of $g_1,\dots,g_n$. Let us extend the functions $f_i$ to $\mathbb R$ by the values $f_i(u)=-\infty$, $u\not\in [0,\overline\alpha^i]$ and put
\begin{align} \label{2.3}
 F(q)& =\sup\{f_1(\alpha_1)+\dots+f_n(\alpha_n):\alpha_1+\dots+\alpha_n=q\}\nonumber\\
     & =-((-f_1)\oplus\dots\oplus(-f_n))(q).
\end{align}
The function $F$ is finite on $[0,\overline q]$, $\overline q=\sum_{i=1}^n\overline\alpha^i$, and takes the value $-\infty$ otherwise. From the properties of an infimal convolution it follows that if $f_i$ are continuous (resp., concave), then $F$ is also continuous (resp., concave): see, e.g., \cite{Str96} (Corollary 2.1 and Theorem 3.1).

Let $q:\mathbb R_+\mapsto [0,\overline q]$ be a measurable function. Consider the equation
\begin{equation} \label{2.4}
 X_t^{x,q}=x+\int_0^t b(X_s^{x,q})\,ds-\int_0^t q_s\,ds
\end{equation}
instead of (\ref{2.1}). If $X_t^{x,q}\ge 0$, then the strategy $q$ is called admissible. The set of such strategies is denoted by $\mathscr A(x)$. Using an appropriate measurable selection theorem (see \cite[Theorem 5.3.1]{Sri98}), we conclude that for any $q\in\mathscr A(x)$ there exists $\alpha\in\mathscr A_n(x)$ such that $F(q_t)=\sum_{i=1}^n f_i(\alpha^i_t)$. It follows that the value function (\ref{2.2}) admits the representation
$$ v(x)=\sup_{q\in\mathscr A(x)} J(x,q),\quad J(x,q)=\int_0^\infty e^{-\beta t} F(q_t)\,dt.$$

Clearly, for any measurable control $q:\mathbb R_+\mapsto [0,\overline q]$ the trajectory $X^{x,q}$ cannot leave the interval $[0,1]$ through the right boundary. Denote by
$$ \tau^{x,q}=\inf\{t\ge 0:X_t^{x,q}=0\}$$
the time of population extinction. As usual, we put $\tau^{x,\alpha}=+\infty$ if $X^{x,\alpha}>0$. Note that $q_t=0$, $t\ge\tau^{x,q}$ for any admissible control $q$.

First, we prove directly that $v$ inherits the concavity property of $f_i$ (see Lemma \ref{lem:2} below).

\begin{lemma} \label{lem:1}
Let $Y$ be a continuous solution of the inequality
$$Y_t\le x+\int_0^t b(Y_s)\,ds-\int_0^t q_s\,ds.$$
Then $Y_t\le X^{x,q}_t$, $t\le\tau:=\inf\{s\ge 0:Y_s=0\}$.
\end{lemma}
\begin{proof}
We follow \cite{BirRot89} (Chapter 1, Theorem 7).
Assume that $Y_{t_1}>X^{x,q}_{t_1}$, $t_1\le\tau$. Let
$t_0=\max\{t\in [0,t_1]:Y_{t}\le X^{x,q}_{t}\}$. We have
\begin{equation} \label{2.5}
Y_{t_0}=X^{x,q}_{t_0},\quad Y_{t}>X^{x,q}_{t},\quad t\in (t_0,t_1].
\end{equation}
The function $Z=Y-X^{x,q}$ satisfies the inequality
$$ 0\le Z_t\le\int_{t_0}^t (b(Y_s)-b(X^{x,q}_s))\,ds\le K \int_{t_0}^t Z_s\,ds,\quad t\in [t_0,t_1],$$
where $K$ is the Lipschitz constant of $b$. By the Gronwall inequality (see, e.g., \cite[Theorem 1.2.1]{Pac98}) we get a contradiction with (\ref{2.5}): $Z_t=0$, $t\in[t_0,t_1]$.
\end{proof}

\begin{lemma} \label{lem:2}
The function $v$ is non-decreasing. If $f_i$ are concave, then $v$ is concave.
\end{lemma}
\begin{proof}
Let $q\in\mathscr A(x)$ and $y>x$. Then
$$ X^{x,q}_t\le y+\int_0^t b(X^{x,q}_s)\,ds-\int_{t_0}^t q_s\,ds.$$
By Lemma \ref{lem:1} we have $X^{x,q}_t\le X^{y,q}_t$ for $t\le\tau^{x,q}$, and hence for all $t\ge 0$. It follows that $\mathscr A(x)\subset\mathscr A(y)$ and $v(x)\le v(y)$.

Let $0\le x^1<x^2\le 1$, $x=\gamma_1 x^1+\gamma_2 x^2$, $\gamma_1,\gamma_2>0$, $\gamma_1+\gamma_2=1$. For $q^i\in\mathscr A(x^i)$ by the concavity of $b$ we have
$$ \gamma_1 X_t^{x^1,q^1}+\gamma_2 X_t^{x^2,q^2}\le x + \int_0^t b(\gamma_1 X_t^{x^1,q^1}+\gamma_2 X_t^{x^2,q^2})\,dt-
   \int_0^t(\gamma_1 q_t^1+\gamma_2 q_t^2)\,dt.$$
Put $q=\gamma_1 q^1+\gamma_2 q^2$.
Applying Lemma \ref{lem:1} to $Y=\gamma_1 X^{x^1,q^1}+\gamma_2 X^{x^2,q^2}$ and $X^{x,q}$ we get the inequality $Y\le X^{x,q}$. It follows that $q\in\mathscr A(x)$. By the concavity of $F$ we obtain:
$$ J(x,q)\ge\int_0^\infty e^{-\beta t}\left(\gamma_1 F(q_t^1)+\gamma_2 F(q_t^2)\right)\,dt=\gamma_1 J(x^1,q^1)+\gamma_2 J(x^2,q^2).$$
It follows that $v$ is concave: $v(x)\ge\gamma_1 v(x^1)+\gamma_2 v(x^2)$.
\end{proof}

Let us introduce the Hamiltonian
\begin{align}
H(x,z) &=b(x) z+\widehat F(z),\nonumber\\
\widehat F(z)&=\sup_{q\in[0,\overline q]}(F(q)-q z)=
\max_{q\in[0,\overline\alpha_1+\dots+\overline\alpha_n]}\max\left\{\sum_{i=1}^n f_i(\alpha_i)-zq:\sum_{j=1}^n\alpha_j=q\right\}\nonumber\\
&=\sum_{i=1}^n\max_{\alpha_i\in[0,\overline\alpha^i]} (f_i(\alpha_i)-z\alpha_i).\label{2.6}
\end{align}
Recall that a continuous function $w:[0,1]\mapsto\mathbb R$ is called a \emph{viscosity subsolution} (resp., a \emph{viscosity supersolution}) of the Hamilton-Jacobi-Bellman (HJB) equation
\begin{equation} \label{2.7}
\beta w(x)-H(x,w'(x))=0
\end{equation}
on a set $K\subset [0,1]$, if for any $x\in K$ and any test function $\varphi\in C^1(\mathbb R)$ such that $x$ is a local maximum (resp., minimum) point of $w-\varphi$, relative to $K$, the inequality
$$ \beta w(x)-H(x,\varphi'(x))\le 0\quad (\textrm{resp.},\ \ge 0)$$
holds true. A function $w\in C([0,1])$ is called a \emph{constrained viscosity solution} (see \cite{Son86}) of (\ref{2.7}) if $u$ is a viscosity subsolution on $[0,1]$ and a viscosity supersolution on $(0,1)$.

By Lemma \ref{lem:2} the value function is continuous. Hence, by Theorem 2.1 of \cite{Son86}, we conclude that $v$ is the unique constrained viscosity solution of (\ref{2.7}). However, in our case it is possible to give a more simple characterization of $v$.
\begin{lemma} \label{lem:3}
Assume that $f_i$ are concave. Then $v$ is the unique continuous function on $[0,1]$, with $v(0)=0$, satisfying the HJB equation (\ref{2.7}) on $(0,1)$ in the viscosity sense.
\end{lemma}
\begin{proof}
Since the equality $v(0)=0$ follows from the definition of $v$, we need only to prove that a continuous function $w$ with $w(0)=0$, satisfying the equation (\ref{2.7}) on $(0,1)$ in the viscosity sense, is uniquely defined. To do this we simply show that $w$ is a viscosity subsolution of (\ref{2.7}) on $[0,1]$ and refer to the cited result of \cite{Son86}.

The inequality
$$ 0=\beta w(0)\le H(0,\varphi'(0))=\widehat F(\varphi'(0))$$
is evident (for any $\varphi\in C^1(\mathbb R)$). Furthermore, in the terminology of \cite[Definitions 2 and 4]{CraNew85}, the point $x=1$ is \emph{irrelevant} and \emph{regular} for the left-hand side of the HJB equation. These properties follow from the fact that $z\mapsto\widehat F(z)$ is non-increasing and $b(1)=0$. By the result of \cite{CraNew85} (Theorem 2), $w$ automatically satisfies the equation (\ref{2.7}) in the viscosity sense on $(0,1]$.
\end{proof}

The subsequent study of the value function strongly relies on its characterization given in Lemma \ref{lem:3}.
Let
\begin{align*}
\partial w(x) &=\{\gamma\in\mathbb R: w(y)-w(x)\ge\gamma(y-x)\},\\
\partial^+ w(x) &=\{\gamma\in\mathbb R: w(y)-w(x)\le\gamma(y-x)\}
\end{align*}
be the sub- and superdifferential of a function $w$. Since $H(x,p)$ is convex in $p$  and satisfies the inequality
$$ |H(x,p)-H(y,p)|=|(b(x)-b(y))p|\le K|p| |x-y|,$$
by \cite[Chapter II, Theorem 5.6]{BarCap97} we infer that
\begin{equation} \label{2.8}
 \beta v(x)-H(x,\gamma)=0,\quad \gamma\in\partial^+ v(x),\quad x\in (0,1).
\end{equation}

As a concave function, $v$ is differentiable on a set $G\subset (0,1)$ with a countable complement $(0,1)\backslash G$. Moreover, $v'$ is continuous and non-increasing on $G$ (see \cite[Theorem 25.2]{Roc70}).
Thus,
\begin{equation} \label{2.9}
 \beta v(x)-H(x,v'(x))=0,\quad x\in G.
\end{equation}

Denote by $\delta_*^i$ the least maximum point  of $f_i$:
$$  \delta_*^i=\min\left(\arg\max_{u\in [0,\overline \alpha^i]} f_i(u)\right).$$
Let us call a strategy $\alpha$ \emph{static} if it does not depend on $t$.
\begin{assumption} \label{as:1}
The static strategy $\delta_*=(\delta^1_*,\dots,\delta^n_*)$ is not admissible for any $x\in [0,1]$. Equivalently, one can assume that $\tau^{x,\delta_*}<\infty$, or $\max_{x\in [0,1]}b(x)<\sum_{i=1}^n \delta_*^i$.
\end{assumption}
In what follows \emph{we suppose that the Assumption \ref{as:1} is satisfied} without further stipulation.

Denote by
$$ v'_+(x)=\lim_{y\downarrow x}\frac{v(y)-v(x)}{y-x},\quad v'_-(x)=\lim_{y\uparrow x}\frac{v(y)-v(x)}{y-x}$$
the right and left derivatives of $v$. It is well known that $\partial^+v(x)=[v'_+(x),v'_-(x)]$, $x\in (0,1)$ and the set-valued mapping $x\mapsto\partial^+ v(x)$ is non-increasing:
\begin{equation} \label{2.10}
\partial^+ v(x)\ge \partial^+ v(y),\quad x<y.
\end{equation}
For $A,B\subset\mathbb R$ we write $A\le B$ if $\xi\le\eta$ for all $\xi\in A$, $\eta\in B$.

\begin{lemma} \label{lem:4}
Assume that $f_i$ are concave. Then the function $v'$ is strictly decreasing on $G$, and $v$ is strictly concave and strictly increasing.
\end{lemma}
\begin{proof}
To prove that $v$ is strictly concave it is enough to show that $x\mapsto \partial^+ v(x)$ is strictly decreasing:
$$ \partial^+ v(x)>\partial^+ v(y),\quad x<y$$
(see \cite[Chapter D, Proposition 6.1.3]{HirUrrLem01}). Assume that $\partial^+ v(x)\cap\partial^+ v(y)\neq\emptyset$, $x<y$. Then the interval $(x,y)$ contains some points $x_1<y_1$, $x_1,y_1\in G$ such that $v'(x_1)=v'(y_1)$. From (\ref{2.10}) it follows that $v'$ is differentiable on $(x_1,y_1)$ and equals to a constant. Differentiating the HJB equation (\ref{2.9}), we get
$$ \beta v'(x)=b'(x) v'(x),\quad x\in (x_1,y_1).$$
Since $b$ is strictly concave, the equality $b'(x)=\beta$, $x\in (x_1,y_1)$ is impossible. Thus, $v'(x)=0$, $x\in (x_1,y_1)$ and
$$\beta v(x)=\widehat F(0)=
\sum_{i=1}^n  f(\delta_*^i),\quad x\in (x_1,y_1).$$

An optimal solution $\alpha^*\in\mathscr A_n(x)$ of the problem (\ref{2.2}) exists (see, e.g., \cite[Theorem 1]{DmiKuz05}).
If $f_i(\alpha_t^{i,*})<f_i(\delta_*^i)=\max_{u\in [0,\overline q^i]}f_i(u)$ on a set of positive measure for at least one index $i$, then
$$ v(x)=J_n(x,\alpha^*)<\sum_{i=1}^n\int_0^\infty e^{-\beta t} f_i(\delta^i_*)\,dt=\frac{1}{\beta}\sum_{i=1}^n f_i(\delta^i_*).$$
If $f_i(\alpha_t^{i,*})=f_i(\delta^i_*)$ a.e., $i=1,\dots,n$, then $\alpha^{i,*}_t\ge \delta_*^i$ a.e. by the definition of $\delta_*$. But this is impossible since the strategy $\delta_*$ is not admissible for $x$ and a fortiori so is $\alpha^*$ (see Lemma \ref{lem:1}).

The obtained contradiction implies that $\partial^+ v$ is strictly decreasing. Hence, $v$ is strictly concave. In view of Lemma \ref{lem:2} this property implies that $v$ is strictly increasing.
\end{proof}

Denote by $g^*(x)=\sup_{y\in\mathbb R}(xy-g(y))$  the Young-Fenchel transform of a function $g:\mathbb R\mapsto (-\infty,\infty]$. Recall (see \cite[Proposition 11.3]{RockWets09}) that for a continuous convex function $g:[a,b]\mapsto\mathbb R$ we have
\begin{equation} \label{2.11}
 \partial g^*(x)=\arg\max_{y\in [a,b]}(xy-g(y)).
\end{equation}

The next result establishes a connection between the differentiability of the value function and the optimality of static strategies.
\begin{lemma} \label{lem:5}
Let $f_i$ be concave. If the value function $v$ is not differentiable at $x_0\in (0,1)$, then the static strategy $q_t=b(x_0)\in\mathscr A(x_0)$ is optimal, and $x_0$ is uniquely defined by the ``golden rule": $b'(x_0)=\beta$.
\end{lemma}
\begin{proof}
Assume that $v'_-(x_0)>v'_+(x_0)$, $x_0\in(0,1)$. By (\ref{2.8}) we have
\begin{equation} \label{2.12}
 \beta v(x_0)=b(x_0)\gamma+\widehat F(\gamma),\quad \gamma\in (v'_+(x_0),v'_-(x_0)).
\end{equation}
Since
\begin{equation} \label{2.13}
\widehat F(z)=\sup_q\{-zq-(-F(q)\}=(-F)^*(-z),
\end{equation}
by (\ref{2.11}), (\ref{2.12}) we obtain
\begin{equation} \label{2.14}
\{\widehat F'(\gamma)\}=\{-b(x_0)\}=-\arg\max_{q\in [0,\overline q]} (F(q)-\gamma q),\quad \gamma\in (v'_+(x_0),v'_-(x_0)).
\end{equation}
Hence, $\widehat F(\gamma)=F(b(x_0))-b(x_0)\gamma$, $\gamma\in (v'_+(x_0),v'_-(x_0))$ and $b(x_0)\in\mathscr A(x_0)$ is optimal:
$$ \beta v(x_0)=F(b(x_0))=\beta J(x_0,b(x_0)).$$

Now assume that the static strategy $b(x_0)$ is optimal. Let us apply the relations Pontryagin's maximum principle to the stationary solution $(X_t,q_t)=(x_0,b(x_0))$ of (\ref{2.4}). Consider the adjoint equation
\begin{equation} \label{2.15}
\dot\psi(t)=-b'(x_0)\psi(t)
\end{equation}
and the basic relation of the Pontryagin maximum principle:
\begin{equation} \label{2.16}
\psi^0 e^{-\beta t} F(b(x_0))=\max_{q\in [0,\overline q]}\left( \psi^0 e^{-\beta t} F(q)+(b(x_0)-q)\psi(t)\right).
\end{equation}
We have $\psi(t)=Ae^{-b'(x_0)t}$ for some $A\in\mathbb R$. If  $(x_0,b(x_0))$ is an optimal solution, then there exist $\psi^0\in\mathbb R_+$, $A\in\mathbb R$ such that $(\psi^0,A)\neq 0$ and the relations (\ref{2.15}), (\ref{2.16}) hold true: see \cite[Theorem 1]{AseKry08}.

Let us rewrite (\ref{2.15}), (\ref{2.16}) as follows
$$ \psi^0  F(b(x_0))=\max_{q\in [0,\overline q]}\left(\psi^0 F(q)+A (b(x_0)-q) e^{(\beta-b'(x_0))t}\right).$$
Assume that $b'(x_0)\neq\beta$. If $\psi^0=0$, then we get a contradiction since $b(x_0)-q$ changes sign on $[0,\overline q]$. Thus, we may assume that $\psi^0=1$:
\begin{align} \label{2.17}
 F(b(x_0)) &=A b(x_0) e^{(\beta-b'(x_0))t}+\max_{q\in [0,\overline q]}\left(F(q) -A e^{(\beta-b'(x_0))t} q\right)\nonumber\\
           &=H(x_0,z_t),\quad z_t=A e^{(\beta-b'(x_0))t}.
\end{align}
But the equality (\ref{2.17}) is impossible, since either $|z_t|\to\infty$ and $H(x_0,z_t)\to+\infty$, $t\to\infty$, or $|z_t|\to 0$ and
$$ H(x_0,z_t)\to H(x_0,0)=\widehat F(0)=\sum_{i=1}^n f_i(\delta^i_*),\quad t\to\infty.$$
In the latter case by (\ref{2.3}) and (\ref{2.17}) we have
$$ F(b(x_0))=\sum_{i=1}^n f_i(\nu_i)=\sum_{i=1}^n f_i(\delta^i_*)$$
for some $\nu_i\in [0,\overline\alpha^i]$ with $\nu_1+\dots+\nu_n=b(x_0)$. From the definition of $\delta^i_*$ it then follows that $\nu_i\ge\delta^i_*$, $i=1,\dots,n$. This is a contradiction, since $\sum_{i=1}^n \delta^i_*\not\in\mathscr A(x_0)$, and  $\sum_{i=1}^n \nu^i=b(x_0)$ should retain this property.
\end{proof}

From the properties of $b$ it follows that either $b'(x)<\beta$, $x\in (0,1)$, or the equation
\begin{equation} \label{2.18}
 b'(x)=\beta,\quad x\in (0,1)
\end{equation}
has a unique solution $\widehat x\in (0,1)$.
\begin{theorem} \label{th:1}
Suppose that $f_i$ are concave. Then the value function $v$ is strictly increasing, strictly concave and continuously differentiable on $(0,1)$, except maybe the point $\widehat x$. If $F$ is differentiable at $b(\widehat x)$, then $v$ is continuously differentiable.
\end{theorem}
\begin{proof}
From Lemma \ref{lem:5} it follows that $\widehat x$ is the only possible discontinuity point of $v$. If $v$  is not differentiable at $\widehat x$, then the interval $(v'_+(\widehat x),v'_-(\widehat x))$ is non-empty. But if $F$ is differentiable at $b(\widehat x)$, then (\ref{2.14}) gives a contradiction: $F'(b(\widehat x))=\gamma$ for all $\gamma\in (v'_+(x_0),v'_-(x_0)).$
\end{proof}

Note that the assumption, concerning the existence of $F'(b(\widehat x))$ is not restrictive. Firstly, $F'$ can have only countably many discontinuity points. Thus, $\widehat x$ is not one of these points for all $\beta\in D$, where $(0,\infty)\backslash D$ is countable. Secondly, the formula
\begin{equation} \label{2.18A}
\partial^+ F(q)=\bigcap_{i=1}^n\partial^+f_i(\alpha^i),\quad\sum_{i=1}^n\alpha^i=q,\quad \sum_{i=1}^n f_i(\alpha^i)=F(q)
\end{equation}
(see \cite[Chapter D, Corollary 4.5.5]{HirUrrLem01}) shows that $F'(b(\widehat x))$ exists if any of the functions $f_i$ is differentiable at $\alpha^i$, satisfying (\ref{2.18A}).

The next result shows that the static strategy $q=b(\widehat x)$ is indeed optimal.
\begin{theorem} \label{th:2}
Assume that $f_i$ are concave. A static strategy $b(x_0)\in\mathscr A(x_0)$, $x_0\in (0,1)$ is optimal if and only if $x_0$ coincides with the solution $\widehat x$ of (\ref{2.18}).
\end{theorem}
\begin{proof} The necessity is proved in Lemma \ref{lem:5}. It remains to prove that $b(\widehat x)\in\mathscr A(\widehat x)$ is optimal. If $v$ is not differentiable at $\widehat x$, the result follows from Lemma \ref{lem:4}. Assume that $v$ is continuously differentiable.

The convex function $\widehat F$ is continuously differentiable on a co-countable set $U\subset\mathbb R$. Furthermore, $v$ is twice differentiable a.e., and $v''\le 0$ a.e., since $v'$ is decreasing. Hence, $\widehat F(v'(x))$ is differentiable on the co-countable set $(v')^{-1}(U)=\{x\in (0,1):v'(x)\in U\}$.
Differentiating the HJB equation (\ref{2.9}),
by the chain rule we obtain
$$ (\beta-b'(x))v'(x)=v''(x) \left(b(x)+\widehat F'(v'(x))\right)\quad a.e.$$
The inequalities
$$\beta-b'(x)<0,\quad x\in (0,\widehat x);\quad \beta-b'(x)>0,\quad x\in (\widehat x,1)$$
imply that $v''(x)<0$ a.e. and
\begin{equation} \label{2.19}
 b(x)+\widehat F'(v'(x))>0,\quad \textrm{a.e. on}\ (0,\widehat x),\quad
 b(x)+\widehat F'(v'(x))<0,\quad \textrm{a.e. on}\ (\widehat x,1).
\end{equation}
Since $v'$ is continuous and strictly decreasing we get the inequalities
$$  b(\widehat x)+\widehat F'_+(v'(\widehat x))\ge 0\ge b(\widehat x)+\widehat F'_-(v'(\widehat x)).$$
Using (\ref{2.11}), (\ref{2.13}), we obtain
\begin{equation} \label{2.20}
 b(\widehat x)\in -\partial \widehat F(v'(\widehat x))=\arg\max_{q\in [0,\overline q]}\{F(q)-v'(\widehat x) q\}.
\end{equation}
It follows that the static strategy $q_t=b(\widehat x)\in\mathscr A(\widehat x)$ is optimal:
$$ \beta v(\widehat x)=b(\widehat x) v'(\widehat x)+\widehat F(v'(\widehat x))=F(b(\widehat x)),\quad v(\widehat x)=J(\widehat x,b(\widehat x)).\qedhere$$
\end{proof}

We turn to the analysis of optimal strategies $q\in\mathscr A(x)$ for $x\neq\widehat x$.
Put
\begin{equation} \label{2.21}
\widehat q(z)=-\partial\widehat F(z).
\end{equation}
On the co-countable set $U$, where $\widehat F$ is differentiable, the mapping (\ref{2.21}) is single-valued.
By (\ref{2.20}) we have
$$ \widehat q(v'(x))=\arg\max_{q\in [0,\overline q]} (F(q)-q v'(x)),\quad v'(x)\in U.$$
Note, that $H_z(x,z)=b(x)-\widehat q(z)$, $z\in U$. From (\ref{2.19}) we know that
$$ H_z(x,v'(x))>0,\quad \textrm{a.e. on}\ (0,\widehat x),\qquad
   H_z(x,v'(x))<0,\quad \textrm{a.e. on}\ (\widehat x,1).$$

We want to use $\widehat q(v'(x))$ as a \emph{feedback control}, formally considering the equation
%\begin{equation} \label{2.20}
$$ \dot X=b(X)-\widehat q(v'(X))=H_z(X,v'(X)),\quad  X_0=x.$$
%\end{equation}
To do it in a rigorous way let us first introduce
$$ \tau^x=\int_x^{\widehat x} \frac{du}{H_z(u,v'(u))}.$$
This definition allows $\tau^x$ to be infinite. Let $x<\widehat x$ (resp., $x>\widehat x$). Then the mapping
$$ \Psi(y)=\int_x^y \frac{du}{H_z(u,v'(u))},\quad \Psi:(x,\widehat x)\mapsto (0,\tau^x)\quad (\textrm{resp.}, \Psi:(\widehat x,x)\mapsto (0,\tau^x))$$
is a bijection.

\begin{lemma} \label{lem:6}
Let $\psi:[a,b]\mapsto\mathbb R$ be continuous and strictly monotonic. Then $\psi^{-1}$ is absolutely continuous if and only if $\psi'\neq 0$ a.e. on $(a,b)$.
\end{lemma}

By Lemma \ref{lem:6}, which proof can be found in \cite{Vill84} (Theorem 2), the equation
\begin{equation} \label{2.22}
 t=\int_x^{Y_t} \frac{du}{H_z(u,v'(u))}
\end{equation}
uniquely defines a locally absolutely continuous function $Y_t$, $t\in (0,\tau^x)$. Moreover, $Y$ is strictly increasing if $x<\widehat x$ and strictly decreasing if $x>\widehat x$. From (\ref{2.22}) we get
\begin{equation} \label{2.23}
\dot Y_t=H_z(Y_t,v'(Y_t))=b(Y_t)-\widehat q(v'(Y_t))\quad \textrm{a.e. on}\  (0,\tau^x),\quad Y_0=x.
\end{equation}
\begin{theorem} \label{th:3}
Let $f_i$ be concave and $x\neq\widehat x$. Put $\mathscr T=\{t\in (0,\tau^x):v'(Y_t)\in U\}$, where $Y$ is defined by (\ref{2.22}). Define the strategy
$$q^*_t=\widehat q(v'(Y_t)),\quad t\in\mathscr T.$$
On the countable set $(0,\tau^x)\backslash\mathscr T$ the values $q^*_t$ can be defined in an arbitrary way. If $\tau^x$ is finite put
$$ q^*_t=b(\widehat x),\quad t\ge\tau^x.$$

The strategy $q^*\in\mathscr A(x)$ is optimal.
\end{theorem}
\begin{proof}
The equality (\ref{2.23}) means that $Y_t=X^{x,q^*}$ on $(0,\tau^x)$.
Furthermore, $X^{x,q^*}=\widehat x$ on $[\tau^x,\infty)$ by the definition of $q^*$. Clearly, $q^*$ is admissible.
To prove that $q^*$ is optimal it is enough to show that
$$ W_t=\int_0^t e^{-\beta s} F(q^*_s)\,ds+e^{-\beta t}v(X_t^{x,q^*})$$
is constant, since then
$$ W_0=v(x)=\lim_{t\to\infty} W_t=\int_0^\infty e^{-\beta s} F(q^*_s)\,ds.$$
We have
\begin{align*}
 \dot W_t &= e^{-\beta t}  F(q^*_t)+e^{-\beta t}\left(-\beta v(X_t^{x,q^*})+v'(X_t^{x,q^*})(b(X_t^{x,q^*})-q^*_t)\right)\\
          &= e^{-\beta t}(-\beta v(X_t^{x,q^*})+H(X_t^{x,q^*},v'(X_t^{x,q^*})))=0\quad \textrm{a.e. on}\  (0,\tau^x).
\end{align*}
For $t>\tau^x$ we have
\begin{align*}
W_t &=\int_0^\tau e^{-\beta s}  F(q^*_s)\,ds+\frac{F(b(\widehat x))}{\beta}(e^{-\beta\tau}-e^{-\beta t})+e^{-\beta t}v(\widehat x)\\
    &=\int_0^\tau e^{-\beta s} F(q_s^*)\,ds +\frac{F(b(\widehat x))}{\beta}e^{-\beta\tau},
\end{align*}
since $v(\widehat x)=F(b(\widehat x))/\beta$ by the optimality of the static strategy $b(\widehat x)$.
\end{proof}

From Theorem \ref{th:3} we see that if the solution $\widehat x$ of (\ref{2.18}) exists, then it attracts any optimal trajectory. Moreover,  $X^{x,q^*}$ is strictly increasing (resp., decreasing) on $(0,\tau^x)$, if $x<\widehat x$ (resp. $x>\widehat x$).

We also mention that the multivalued feedback control $\widehat q(v'(x))$ satisfies the inequalities
\begin{equation} \label{2.24}
 b(x)>\widehat q(v'(x)),\quad x\in (0,\widehat x);\quad b(x)<\widehat q(v'(x)),\quad x\in (\widehat x,1).
\end{equation}
Indeed, $\widehat q(z)=-\partial F(z)$ is a non-increasing multivalued mapping. On a co-countable set $U$ the mappings $\widehat q(v'(x))$ are single-valued, non-decreasing and satisfy the inequalities (\ref{2.19}). Thus, in any neighbourhood of a point $x\neq\widehat x$ there exist $x_1<x$, $x_2>x$ such that
$$ \widehat q(v'(x_1))\le\widehat q(v'(x))\le\widehat q(v'(x_2)),$$
where $\widehat q(v'(x_i))$ are single-valued and satisfy (\ref{2.19}). It easily follows that
\begin{equation} \label{2.25}
 b(x)\ge\widehat q(v'(x)),\quad x\in (0,\widehat x);\quad b(x)\le\widehat q(v'(x)),\quad x\in (\widehat x,1).
\end{equation}
Assume that $b(x_0)\in\widehat q(v'(x_0))$, $x_0\neq\widehat x$. Then from the HJB equation (\ref{2.9}) it follows that $q=b(x_0)\in\mathscr A(x_0)$ is an optimal strategy: $\beta v(x_0)=F(b(x_0))$, in contradiction with Lemma \ref{lem:5}. Thus, the inequalities (\ref{2.25}) are strict.

\section{Cooperative harvesting problem: the case of non-concave revenues}
\label{sec:3}
\setcounter{equation}{0}
Now we drop the assumption that $f_i$ are concave. Let us extend the class of harvesting strategies. A family $(\mu_t(dx))_{t\ge 0}$ of probability measures on $[0,\overline q]$ is called a \emph{relaxed control} if the function
$$t\mapsto \int_0^{\overline q} \varphi(y)\,\mu_t(dy)$$
is measurable for any continuous function $\varphi$. A relaxed control $\mu$ induces the dynamics
$$ X_t=x+\int_0^t b(X_s)\,ds-\int_0^t \int_0^{\overline q} y \mu_s(dy)\,ds.$$
The related value function is defined as follows
\begin{equation} \label{3.1}
 v_r(x)=\sup_{\mu\in\mathscr A^r(x)} J^r(x,\mu),\quad
  J^r(x,\mu)=\int_0^\infty e^{-\beta t}\int_0^{\overline q} F(y) \mu_t(dy)\,dt,\quad x\in [0,1],
\end{equation}
where $\mathscr A^r=\{\mu:X^{x,\mu}\ge 0\}$ is the class of admissible relaxed controls.

 Denote by $\widetilde F$ the concave hull of $F$: $\widetilde F=-(-F)^{**}$. Let
\begin{equation} \label{3.2}
\widetilde v(x)=\sup_{q\in\mathscr A(x)}\widetilde J(x,q),\quad \widetilde J(x,q)=\int_0^\infty e^{-\beta t} \widetilde F(q_t)\,dt
\end{equation}
be the related value function. Note that by (\ref{2.3}) and the properties of infimal convolution (\cite{IofTih79}, Chapter 3, \S\,3.4, Theorem 1) we have
$$ -\widetilde F=(-F)^{**}=(-f_1)^{**}\oplus\dots\oplus(-f_n)^{**}=(-\widetilde f_1)\oplus\dots\oplus(-\widetilde f_n),$$
where $\widetilde f_i$ is the convex hull of $f_i$. Hence,
\begin{equation} \label{3.3}
\widetilde F(q)=\sup\{\widetilde f_1(\alpha_1)+\dots+\widetilde f_n(\alpha_n):\alpha_1+\dots+\alpha_n=q\}.
\end{equation}

Since $\widetilde F\ge F$ it follows that $\widetilde v\ge v$.
By the Jensen inequality we have
$$ J^r(x,\mu)\le \int_0^\infty e^{-\beta t}\int_0^{\overline q} \widetilde F(y) \mu_t(dy)\,dt
\le\int_0^\infty e^{-\beta t} \widetilde F(\widetilde q_t)\,dt,$$
where $q_t=\int_0^{\overline q} y\mu_t(dy)$ is an admissible control for the problem (\ref{2.4}). Thus,
$$ v(x)\le v_r(x)\le \widetilde v(x).$$

\begin{lemma} \label{lem:7}
For any $p\in [0,\overline q]$ there exists $p_1, p_2\in (0,1)$, $\varkappa\in (0,1)$ such that
$$ p=\varkappa p_1+(1-\varkappa) p_2,\quad \widetilde F(p)=\varkappa F(p_1)+(1-\varkappa) F(p_2).$$
\end{lemma}
The proof of a more general result can be found in \cite{HirUrrLem01} (Chapter E,  Proposition 1.3.9(ii)).

Denote by $\widetilde q_t$ the strategy, constructed in Theorem \ref{th:3}, where $F$ is replaced by $\widetilde F$.
We claim that
\begin{equation} \label{3.4}
\widetilde F(\widetilde q_t)=F(\widetilde q_t),\quad \textrm{a.e. on } (0,\tau^x).
\end{equation}
By construction, $\widetilde q_t$
is the unique maximum point of $q\mapsto \widetilde F(q)-qv'(Y_t)$ on $[0,\overline q]$ for all $t\in \widetilde{\mathscr T}$, where $(0,\tau^x)\backslash\widetilde{\mathscr T}$ is countable. If $\widetilde F(\widetilde q_t)\neq F(\widetilde q_t)$, $t\in \widetilde{\mathscr T}$ then, by Lemma \ref{lem:7}, $\widetilde F$ is affine in an open neighbourhood of $\widetilde q_t$, and
$$\arg\max_{q\in [0,\overline q]}(\widetilde F(q)-v'(Y_t)q\}$$
contains this neighbourhood: a contradiction.

Furthermore, by Lemma \ref{lem:7} there exist $p_1,p_2\in [0,1]$, $\varkappa\in (0,1)$ such that
\begin{equation} \label{3.5}
 b(\widehat x)=\varkappa p_1+(1-\varkappa) p_2,\qquad \widetilde F(b(\widehat x))=\varkappa F(p_1)+(1-\varkappa) F(p_2).
\end{equation}
Consider the static relaxed control
\begin{equation} \label{3.6}
\mu_s=\begin{cases}
\widetilde q_s,& s<\tau^x,\\
\varkappa\delta_{p_1}+(1-\varkappa)\delta_{p_2},& s\ge\tau^x,
\end{cases}
\end{equation}
where $\delta_a$ is the Dirac measure, concentrated at $a$. By (\ref{3.4}), (\ref{3.5}) we have
$$ J^r(x,\mu)=\int_0^{\tau^x} e^{-\beta t} F(\widetilde q_t)\,dt+\int_{\tau^x}^\infty e^{-\beta t} (\varkappa F(p_1)+(1-\varkappa)F(p_2))\,dt=\widetilde J(x,\widetilde q).$$
Thus, $v_r(x)=\widetilde v(x)$ and the strategy (\ref{3.6}) is optimal for the relaxed problem (\ref{3.1}).

To prove that $v_r(x)=v(x)$ let us construct an approximately optimal strategy
\begin{equation} \label{3.7}
q^\varepsilon\in\mathscr A(x): J(x,q^\varepsilon)\to v_r(x),\quad \varepsilon\to 0.
\end{equation}
We may assume that $p_1\neq p_2$ and $p_1<b(\widehat x)< p_2.$
Otherwise, the strategy (\ref{3.6}) reduces to an ordinary control $\mu_s=\widetilde q_s I_{\{s<\tau^x\}}+b(\widehat x) I_{\{s\ge\tau^x\}}$ and we conclude that $v(x)=v_r(x)=\widetilde v(x)$.

Define $g$ by the equation
\begin{align} \label{3.8}
 & \int_{\widehat x-\varepsilon}^{\widehat x}(b(\widehat x)-b(x))\rho(x)\,dx =\int_{\widehat x}^{\widehat x+g(\varepsilon)}
 (b(x)-b(\widehat x))\rho(x)\,dx,\\
 & \rho(x)  =\frac{1}{(b(x)-p_1)(p_2-b(x))}.\nonumber
\end{align}
Note, that for sufficiently small $\varepsilon>0$ we have $\rho(x)>0$ on $(\widehat x-\varepsilon,g(\varepsilon))$ and integrands in (\ref{3.8}) are positive. Clearly, $g(\varepsilon)\downarrow 0$, $\varepsilon\to 0$. Put
\begin{align*}
  \tau_1 & =\int_{\widehat x}^{\widehat x+g(\varepsilon)}\frac{dx}{b(x)- p_1},\quad
   \tau_2=\int_{\widehat x-\varepsilon}^{\widehat x+g(\varepsilon)}\frac{dx}{p_2-b(x)},\\
   \tau_3 &=\int_{\widehat x-\varepsilon}^{\widehat x}\frac{dx}{b(x)-p_1},\quad
   \tau=\tau_1+\tau_2+\tau_3.
\end{align*}
For brevity, we omit the dependence of $\tau_i$ on $\varepsilon$. Put
\begin{equation} \label{3.9}
 q^\varepsilon_t=\sum_{j=0}^\infty \left(p_1 I_{[j\tau,j\tau+\tau_1)}(t)+p_2 I_{[j\tau+\tau_1,j\tau+\tau_1+\tau_2)}(t) + p_1 I_{[j\tau+\tau_1+\tau_2,(j+1)\tau)}(t)\right).
\end{equation}

The trajectory $X^{\widehat x,q^\varepsilon}$ is periodic:
\begin{align*}
 \dot X^{\widehat x,q^\varepsilon}_t&=b(X^{\widehat x,q^\varepsilon}_t)- p_1,\quad (j\tau,j\tau+\tau_1),
 \quad X^{\widehat x,q^\varepsilon}_{j\tau}=\widehat x,\\
 \dot X^{\widehat x,q^\varepsilon}_t&=b(X^{\widehat x,q^\varepsilon}_t)- p_2,\quad (j\tau+\tau_1,j\tau+\tau_1+\tau_2),
  \quad X^{\widehat x,q^\varepsilon}_{j\tau+\tau_1}=\widehat x+g^\varepsilon,\\
 \dot X^{\widehat x,q^\varepsilon}_t&=b(X^{\widehat x,q^\varepsilon}_t)- p_1,\quad (j\tau+\tau_1+\tau_2,(j+1)\tau)),
  \quad X^{\widehat x,q^\varepsilon}_{j\tau+\tau_1+\tau_2}=\widehat x-\varepsilon.
\end{align*}
 It sequentially visits the points $\widehat x$, $\widehat x+g^\varepsilon$, $\widehat x-\varepsilon$, $\widehat x$ and moves monotonically between them. Furthermore,
\begin{align*}
\int_{j\tau}^{(j+1)\tau} e^{-\beta t} F(q_t^\varepsilon)\,dt &=\frac{e^{-\beta j\tau}}{\beta}
\left((1-e^{-\beta \tau_1})F(p_1)+
(e^{-\beta \tau_1}-e^{-\beta (\tau_1+\tau_2)})F(p_2)\right.\\
&\left.+(e^{-\beta (\tau_1+\tau_2)}-e^{-\beta \tau})F(p_1)\right)
\end{align*}
Thus,
\begin{align*}
J(\widehat x,q^\varepsilon)&=\frac{1}{\beta(1-e^{-\beta\tau})}\left((1-e^{-\beta \tau_1})F(p_1)+
(e^{-\beta \tau_1}-e^{-\beta (\tau_1+\tau_2)})F(p_2)\right.\\
&\left.+(e^{-\beta (\tau_1+\tau_2)}-e^{-\beta \tau})F(p_1)\right)=\frac{1}{\beta}\left(\frac{\tau_1+\tau_3}{\tau}F(p_1) +\frac{\tau_2}{\tau}F(p_2)\right)+o(1),\quad\varepsilon\to 0.
\end{align*}
Since
$$ \tau_1=\frac{g(\varepsilon)}{b(\widehat x)-p_1}(1+o(1)),\quad \tau_2=\frac{g(\varepsilon)+\varepsilon}{p_2-b(\widehat x)}(1+o(1)),\quad \tau_3=\frac{\varepsilon}{b(\widehat x)-p_1}(1+o(1)),$$
using (\ref{3.5}), we get
$$ \frac{\tau_1+\tau_3}{\tau_2}=\frac{p_2-b(\widehat x)}{b(\widehat x)-p_1}=\frac{\varkappa}{1-\varkappa},$$
$$ \frac{\tau_1+\tau_3}{\tau}=\frac{1}{1+\tau_2/(\tau_1+\tau_3)}=\varkappa,\qquad
 \frac{\tau_2}{\tau}=\frac{1}{1+(\tau_1+\tau_3)/\tau_2}=1-\varkappa.$$
Thus,
$$\lim_{\varepsilon\to 0}J(\widehat x,q^\varepsilon)=\frac{1}{\beta}(\varkappa F(p_1)+(1-\varkappa) F(p_2))=\frac{\widetilde F(b(\widehat x))}{\beta}=v(\widehat x).$$

We see that the strategy (\ref{3.9}) satisfies (\ref{3.7}), and $v(x)=v_r(x)=v(x)$. The obtained results are summarized below.
\begin{theorem}
%Assume that the functions $f_i$ are continuous and Assumption \ref{as:1} is satisfied.
The value functions (\ref{2.2}), (\ref{3.1}), (\ref{3.2}) coincide: $v=v_r=\widetilde v$. By Theorem \ref{th:1}, applied to (\ref{3.2}), $v$  is strictly increasing, strictly concave and continuously differentiable on $(0,1)$, except maybe the point $\widehat x$. If $\widetilde F$ is differentiable at $b(\widehat x)$, then $v$ is continuously differentiable. The strategy (\ref{3.6}) is optimal for the relaxed problem (\ref{3.1}).
\end{theorem}

\section{Rational taxation}
\label{sec:4}
\setcounter{equation}{0}
Assume that a regulator imposes the proportional tax $v'(x)\alpha$ for the fishing intensity $\alpha$. Then the myopic agents take their optimal strategies from the sets
$$ \widehat\alpha^i(x)=\arg\max_{u\in [0,\overline \alpha^i]}\{f_i(u) -v'(x) u\}.$$
The direct implementation of such feedback controls may cause technical problems, since the related equation (\ref{2.1})
can be unsolvable. Instead of continuous change of the tax $v'(X_t)$, a more realistic approach   consists in its fixing for some periods of time: $v'(X_{\tau_j})$, $t\in [\tau_j,\tau_{j+1})$. In this case agents also fix their strategies:
$$\alpha^i_{\tau_i}\in\arg\max_{u\in [0,\overline \alpha^i]}\{f_i(u) -v'(X_{\tau_j}) u\},\quad t\in [\tau_j,\tau_{j+1}).$$
This scheme results in ``step-by-step positional control'' (see \cite{KraSub88}), defined recursively by the formulas:
\begin{align}
 X^{x,\alpha}_0&=x,\nonumber\\
 \alpha_t^i&=\alpha_{\tau_j}^i\in\arg\max_{u\in [0,\overline \alpha^i]}\{f_i(u) -v'(X_{\tau_j}^{x,\alpha}) u\},\quad  t\in [\tau_j,\tau_{j+1}),\label{4.1}\\
 X_t^{x,\alpha} &=X_{\tau_j}^{x,\alpha}+\int_{\tau_j}^t b(X_s^{x,\alpha})\,ds- \sum_{i=1}^n \alpha^i_{\tau_j}\cdot(t-\tau_j),\quad t\in [\tau_j,\tau_{j+1}),\nonumber\\
 0&=\tau_0<\dots\tau_j<\dots,\quad \tau_j\to\infty,\quad j\to\infty, \label{4.2}
\end{align}
bypassing at the same time the mentioned technical problems.

\begin{theorem}
Let $\widetilde F'(\widehat x)$ exist. Then for any $\varepsilon>0$, $\delta>0$ there exists a sequence (\ref{4.2}) such that the strategy (\ref{4.1}) is approximately optimal: $J_n(x,\alpha)\ge v(x)-\varepsilon$ and stabilizing in the following sense:
$$|X_t^{x,\alpha}-\widehat x|<\delta,\quad t\ge \overline t(x,\varepsilon,\delta).$$
\end{theorem}
\begin{proof}
First note that
$$ \widehat\alpha^i(z):=\arg\max_{u\in[0,\overline\alpha^i]}(f_i(u)-zu)\subset
     \widetilde\alpha^i(z):=\arg\max_{u\in[0,\overline\alpha^i]}(\widetilde f_i(u)-zu).$$
Indeed, if $u^*\in\widehat\alpha^i(z)$, then $-z\in\partial(-f_i)(u^*)$ and $u^*\in\partial (-f_i)^*(-z)$: see \cite[Chapter E, Proposition 1.4.3]{HirUrrLem01}. But, by (\ref{2.11}),
$$\partial (-f_i)^*(-z)=\arg\max_{u\in[0,\overline\alpha^i]}(-zu-(-f_i)^{**}(u))=\arg\max_{u\in[0,\overline\alpha^i]}(\widetilde f_i(u)-zu)=\widetilde\alpha^i(z).$$

Furthermore, from the representation (\ref{3.3}) we get
$$\max_{q\in[0,\overline q]}\{\widetilde F(q)-zq\}
=\sum_{i=1}^n\max_{\alpha_i\in[0,\overline\alpha^i]}\{\widetilde f_i(\alpha_i)-z\alpha_i\}
$$
(see also (\ref{2.6})). Thus,
\begin{equation} \label{4.3}
\widetilde q(z):=\arg\max_{q\in[0,\overline q]}(\widetilde F(q)-zq)=\sum_{i=1}^n\widetilde\alpha^i(z) \supset\sum_{i=1}^n\widehat\alpha^i(z).
\end{equation}
From (\ref{2.24}) it then follows that
\begin{equation} \label{4.4}
 b(x)>\sum_{i=1}^n\widehat\alpha^i(v'(x)),\quad x\in(0,\widehat x),\quad
 b(x)<\sum_{i=1}^n\widehat\alpha^i(v'(x)),\quad x\in(\widehat x,1).
\end{equation}

The subsequent argumentation follows the introductory section of \cite{IshKoi00}.
For any $x_0\in(0,1)$ and any $\alpha_0^i\in\widehat\alpha^i(v'(x_0))$ we have
$$\beta v(x_0)=\left(b(x_0)-\sum_{i=1}^n\alpha_0^i\right)v'(x_0)+\sum_{i=1}^n f_i(\alpha_0^i).$$
Put, $$\psi (x,\alpha)=-\beta v(x)+\left(b(x)-\sum_{i=1}^n\alpha^i\right)v'(x)+\sum_{i=1}^n f_i(\alpha^i)$$
and define the time moment
\begin{align}
 \tau_1 &=\inf\{t\ge 0:\psi(X_t^{x_0,\alpha_0},\alpha_0)<-\beta\varepsilon\ \textrm{or } X_t^{x_0,\alpha_0}>\widehat x+\delta\}, \quad x_0\in (0,\widehat x),\label{4.5}\\
 \tau_1 &=\inf\{t\ge 0:\psi(X_t^{x_0,\alpha_0},\alpha_0)<-\beta\varepsilon\ \textrm{or } X_t^{x_0,\alpha_0}<\widehat x-\delta\}, \quad x_0\in (\widehat x,1),\label{4.6}\\
 \tau_1 &=\inf\{t\ge 0:\psi(X_t^{x_0,\alpha_0},\alpha_0)<-\beta\varepsilon\ \textrm{or } X_t^{x_0,\alpha_0}\not\in(\widehat x-\delta,\widehat x+\delta)\},\quad x_0=\widehat x.\label{4.7}
\end{align}
For $t\in [0,\tau_1]$ in each of the cases (\ref{4.5}), (\ref{4.6}), (\ref{4.7}) we have respectively
$$X_t^{x_0,\alpha_0}\in [x_0,\widehat x+\delta],\quad X_t^{x_0,\alpha_0}\in [\widehat x-\delta,x_0],\quad
X_t^{x_0,\alpha_0}\in [\widehat x-\delta,\widehat x+\delta].$$
Assume that $x_{k-1}$, $\alpha_{k-1}$, $\tau_k$ are defined. Put $$x_k=X_{\tau_k}^{x_{k-1},\alpha_{k-1}},\quad \alpha_k^i\in\widehat\alpha^i(v'(x_k)),$$
\begin{align}
 \tau_{k+1} &=\inf\{t\ge\tau_k:\psi(X_t^{x_k,\alpha_k},\alpha_k)<-\beta\varepsilon\ \textrm{or } X_t^{x_k,\alpha_k}>\widehat x+\delta\}, \quad x_k\in (0,\widehat x),\label{4.8}\\
 \tau_{k+1} &=\inf\{t\ge\tau_k:\psi(X_t^{x_k,\alpha_k},\alpha_k)<-\beta\varepsilon\ \textrm{or } X_t^{x_k,\alpha_k}<\widehat x-\delta\}, \quad x_k\in (\widehat x,1),\label{4.9}\\
 \tau_{k+1} &=\inf\{t\ge\tau_k:\psi(X_t^{x_k,\alpha_k},\alpha_k)<-\beta\varepsilon\ \textrm{or } X_t^{x_k,\alpha_k}\not\in(\widehat x-\delta,\widehat x+\delta)\},\quad x_k=\widehat x.\label{4.10}
\end{align}

The function $x\mapsto\psi(x,\alpha)$ is uniformly continuous on any interval $[a,b]\subset(0,1)$ uniformly in $\alpha\in [0,\overline q]$. Thus, there exists $\delta'$ such that if
$$|\psi(x,\alpha)-\psi(y,\alpha)|\ge\beta\varepsilon,\quad [x,y]\subset [a,b],$$
then $|x-y|\ge\delta'$. Assume that $\psi(X_{\tau_{k+1}}^{x_k,\alpha_k},\alpha_k)=-\beta\varepsilon$. Since $\psi(x_k,\alpha_k)=0$, we get
$$\delta'\le |X_{\tau_{k+1}}^{x_k,\alpha_k}-x_k|\le\int_{\tau_k}^{\tau_{k+1}} b(X_t^{x_k,\alpha_k})\,dt+\int_{\tau_k}^{\tau_{k+1}}\sum_{i=1}^n\alpha_k^i\,dt \le(\overline b+\overline q)(\tau_{k+1}-\tau_k),$$
where $\overline b=\max_{x\in [0,1]} b(x)$.
Furthermore, if $\psi(X_{\tau_{k+1}}^{x_k,\alpha_k})>-\beta\varepsilon$ and $\tau_{k+1}<\infty$, then in any of three cases (\ref{4.8}), (\ref{4.9}), (\ref{4.10}) we have
$$\delta\le |X_{\tau_{k+1}}^{x_k,\alpha_k}-x_k|\le (\overline b+\overline q)(\tau_{k+1}-\tau_k).$$
Thus, the differences $\tau_{k+1}-\tau_k$ are uniformly bounded from below by a positive constant, and the strategy
$\alpha=\sum_{k=0}^\infty\alpha_k I_{[\tau_k,\tau_{k+1})}(t)$ is well defined for all $t\ge 0$. Note, that $X^{x_0,\alpha}_t$ belongs to one of the sets $[x_0,\widehat x+\delta]$, $[\widehat x-\delta,x_0]$, $[\widehat x-\delta,\widehat x+\delta]$ for all $t\ge 0$.

By the Berge maximum theorem (see \cite[Theorem 17.31]{AliBor06}) the set-valued mapping $\widehat\alpha$ is upper hemicontinuous, hence its graph is closed (see \cite[Theorem 17.10]{AliBor06}). From (\ref{4.4}) it then follows that there is a finite gap between $b(x)$ and $\sum_{i=1}^n\widehat\alpha^i(v'(x))$ on $(0,\widehat x-\delta)\cup(\widehat x+\delta,1)$. Thus, $|\dot X^{\alpha,x_0}|$ is uniformly bounded from below by a positive constant, when $X^{\alpha,x_0}\in (0,\widehat x-\delta)\cup(\widehat x+\delta,1)$. This property implies that $X^{\alpha,x_0}$ reaches the neighbourhood $[\widehat x-\delta,\widehat x+\delta]$ in finite time $\overline t(x,\varepsilon,\delta)$. After reaching this neighbourhood, $X^{\alpha,x_0}$ remains in it forever by the construction of $\alpha$.

It remains to prove that $\alpha$ is $\varepsilon$-optimal. We have
$$ -\beta v(X_t^{x_k,\alpha_k})+\left(b(X_t^{x_k,\alpha_k})-\sum_{i=1}^n\alpha_k^i\right)v'(X_t^{x_k,\alpha_k})+\sum_{i=1}^n f_i(\alpha_k^i)\ge-\beta\varepsilon,\quad t\in (\tau_k,\tau_{k+1}).$$
After the multiplication on $e^{-\beta t}$ an integration we get
$$ e^{-\beta\tau_{k+1}} v(X_{\tau_{k+1}}^{x_k,\alpha_k})-e^{-\beta\tau_k} v(X_{\tau_k}^{x_k,\alpha_k})+\int_{\tau_k}^{\tau_{k+1}} e^{-\beta t}\sum_{i=1}^n f_i(\alpha_k^i)\,dt\ge\varepsilon(e^{-\beta\tau_{k+1}}-e^{-\beta\tau_k}).$$
Summing up and passing to the limit we obtain the desired inequality:
$$ \int_0^\infty e^{-\beta t}\sum_{i=1}^n f_i(\alpha_t^i)\,dt\ge v(x_0)-\varepsilon. \qedhere$$
\end{proof}

As an example, consider the problem with $n$ identical agents and assume that their common profit function is linear: $f_i(u)=f(u)=u$, $u\in[0,\overline\alpha]$. The HJB equation (\ref{2.9}) takes the form
$$\beta v(x)=b(x)v'(x)+ n \max_{u\in [0,\overline\alpha]}(u-v'(x)u).$$
From (\ref{2.20}) it follows that $v'(\widehat x)=1$. Thus,
\begin{equation} \label{4.11}
v'(x)>1,\quad x<\widehat x,\quad v'(x)<1,\quad x>\widehat x
\end{equation}
and $v$ satisfies the equations
$$ \beta v(x)=b(x)v'(x),\quad x<\widehat x;\qquad \beta v(x)=(b(x)-n\overline\alpha)v'(x)+ n\overline\alpha,\quad x>\widehat x.$$
Solving these equations, by the uniqueness result, given in Lemma \ref{lem:3}, we infer that
$$ v(x)=\frac{b(\widehat x)}{\beta}\exp\left(-\int_x^{\widehat x}\frac{\beta}{b(y)}\,dy\right),\quad x\in (0,\widehat x],$$
$$ v(x)=\frac{1}{\beta}(b(\widehat x)-n\overline\alpha)\exp\left(\int_{\widehat x}^x\frac{\beta}{b(y)-\overline\alpha n}\,dy\right)+\frac{1}{\beta} n\overline \alpha,\quad x\in [\widehat x,1].$$
For the biomass quantities $x$ below the critical level $\widehat x$ the tax $v'(x)$ does not depend on $n$:
$$ v'(x)=\frac{b(\widehat x)}{b(x)}\exp\left(-\int_x^{\widehat x}\frac{\beta}{b(y)}\,dy\right),\quad x\in (0,\widehat x].$$
For larger values of $x$ we have
$$ v'(x)=\frac{n\overline\alpha-b(\widehat x)}{n\overline\alpha-b(x)}\exp\left(-\int_{\widehat x}^x\frac{\beta}{n\overline\alpha-b(y)}\,dy\right),\quad x\in [\widehat x,1].$$
In particular, $v'(x)\to f'(0)=1$, $n\to\infty$.

Note, that a tax, stimulating an optimal cooperative behavior is by no means unique. For instance, any tax, satisfying (\ref{4.11}), can serve this purpose. So, the most interesting quantity is the ``critical tax''
\begin{equation} \label{4.12}
v'(\widehat x)=\widetilde F'(b(\widehat x)).
\end{equation}
The equality (\ref{4.12}) follows from (\ref{2.20}). Consider
$\widetilde F$ as the value function of the elementary problem (\ref{3.3}),
where the artificial agents with concave revenues $\widetilde f_i$ cooperatively distribute some given harvesting intensity $q$. Formula (\ref{4.12}) shows that $v'(\widehat x)$ is simply the shadow price of the critical growth growth rate $b(\widehat x)$ within this problem.

We are interested in the dependence of the critical tax $v'(\widehat x)$ on the size of agent community. Consider again $n$ identical agents with the revenue functions $f_i=f$. If $f$ is linear, the critical tax, as we have seen, does not depend on $n$. Assume now that $f$ is differentiable and strictly concave. Then by (\ref{2.20}) and (\ref{4.3}) we get
$$b(\widehat x)\in \sum_{i=1}^n\arg\max_{u\in[0,\overline\alpha]}\{f(u)-v'(\widehat x)u\}$$
Taking optimal values of $u$ to be equal, we conclude that
$v'(\widehat x)=f'(b(\widehat x)/n)$. Thus, $v'(\widehat x)$ is increasing in $n$, and $v'(\widehat x)\to f'(0)$, $n\to\infty$. Our final result shows that this situation is typical: the critical tax can only increase, when the agent community widens.
\begin{theorem}
Denote by $F_n$, $F_{n+m}$ and $v_n$, $v_{n+m}$ the cooperative instantaneous revenue functions (\ref{2.3}) and the value functions (\ref{2.2}), corresponding to the agent communities
$$ \{f_i\}_{i=1}^n\subset\{f_i\}_{i=1}^{n+m}.$$
Assume that $\widetilde F'_n(b(\widehat x))$, $\widetilde F'_{n+m}(b(\widehat x))$ exist. Then
$$ v'_n(\widehat x)=\widetilde F'_n(b(\widehat x))\le v'_{n+m}(\widehat x)=\widetilde F'_{n+m}(b(\widehat x)).$$
\end{theorem}
\begin{proof}
It is enough to consider the case $m=1$. By the associativity of the infimal convolution we have
$$ (-\widetilde F_{n+1})(q)=(-\widetilde F_n)\oplus(-\widetilde f_{n+1})(q).$$
The formula for the subdifferential of an infimal convolution, given in \cite[Chapter D, Corollary 4.5.5]{HirUrrLem01}, implies that
$$ \partial(-\widetilde F_{n+1})(q)\subseteq\bigcup_{u}\partial (-\widetilde F_n)(u)\cap\partial(-\widetilde f_{n+1})(q-u)\subseteq\bigcup_{u\in [0,q]}\partial (-\widetilde F_n)(u).$$
But since the set-valued mapping $u\mapsto\partial (-\widetilde F_{n+1})(u)$ is non-decreasing, we have
$$  \partial(-\widetilde F_{n+1})(q)\le\partial (-\widetilde F_n)(q),\quad q\in [0,\overline q].$$
Thus, $\widetilde F'_{n+1}(b(\widehat x))\ge \widetilde F'_n(b(\widehat x))$.
\end{proof}

A resembling result for discrete time problem was proved in \cite[Theorem 3]{Rok00}.

  \bibliographystyle{plain}
  \bibliography{litFish}

\begin{thebibliography}{10}

\bibitem{AliBor06}
C.D. Aliprantis and K.C. Border.
\newblock {\em Infinite dimensional analysis. A hitchhiker's guide}.
\newblock Springer, Berlin, 2006.

\bibitem{Arn09}
R.~Arnason.
\newblock Fisheries management and operations research.
\newblock {\em Eur. J. Oper. Res.}, 193(3):741--751, 2009.

\bibitem{AseKry08}
S.M. Aseev and A.V. Kryazhimskii.
\newblock On a class of optimal control problems arising in mathematical
  economics.
\newblock {\em Proc. Steklov Inst. Math.}, 262(1):10--25, 2008.

\bibitem{BarCap97}
M.~Bardi and I.~Capuzzo-Dolcetta.
\newblock {\em Optimal control and viscosity solutions of
  {H}amilton-{J}acobi-{B}ellman equations}.
\newblock Birkhauser, Boston, 1997.

\bibitem{BirRot89}
G.~Birkhoff and G.-C. Rota.
\newblock {\em Ordinary differential equations}.
\newblock Wiley, New York, 1989.

\bibitem{Cla79}
C.W. Clark.
\newblock Mathematical models in the economics of renewable resources.
\newblock {\em SIAM Rev.}, 21(1):81--99, 1979.

\bibitem{Cla80}
C.W. Clark.
\newblock Towards a predictive model for the economic regulation of commercial
  fisheries.
\newblock {\em Can. J. Fish. Aquat. Sci.}, 37(7):1111--1129, 1980.

\bibitem{Cla06}
C.W. Clark.
\newblock {\em The worldwide crisis in fisheries. Economic models and human
  behavior}.
\newblock Cambridge University Press, Cambridge, 2006.

\bibitem{CraNew85}
M.G. Crandall and R.~Newcomb.
\newblock Viscosity solutions of {H}amilton-{J}acobi equations at the boundary.
\newblock {\em Proc. Amer. Math. Soc.}, 94(2):283--2903, 1985.

\bibitem{DmiKuz05}
A.V. Dmitruk and N.V. Kuz'kina.
\newblock Existence theorem in the optimal control problem on an infinite time
  interval.
\newblock {\em Math. Notes}, 78(4):466--480, 2005.

\bibitem{Gor54}
H.S. Gordon.
\newblock The economic theory of a common-property resource: the fishery.
\newblock {\em J. Polit. Econ.}, 62(2):124--142, 1954.

\bibitem{HanShoWhi97}
N.~Hanley, J.F. Shogren, and B.~White.
\newblock {\em Environmental Economics in Theory and Practice}.
\newblock Macmillan Education UK, London, 1997.

\bibitem{Har68}
G.~Hardin.
\newblock The tragedy of the commons.
\newblock {\em Science}, 162:1243--1248, 1968.

\bibitem{HirUrrLem01}
J.-B. Hiriart-Urruty and C.~Lemar{\'e}chal.
\newblock {\em Fundamentals of convex analysis}.
\newblock Springer, Berlin, 2001.

\bibitem{Il09}
V.G. Il'ichev.
\newblock {\em Stability, adaptation and contol in ecological systems}.
\newblock Fizmatlit, Moscow, 2009.

\bibitem{IofTih79}
A.D. Ioffe and V.M. Tihomirov.
\newblock {\em Theory of extremal problems}.
\newblock North-Holland, Amsterdam, 1979.

\bibitem{IshKoi00}
H.~Ishii and S.~Koike.
\newblock On $\varepsilon$-optimal controls for state constraint problems.
\newblock {\em Ann. Inst. H. Poincar\'{e}, Anal. Non Lin\'{e}aire},
  17(4):473--502, 2000.

\bibitem{KraSub88}
N.N. Krasovskii and A.I. Subbotin.
\newblock {\em Game-theoretical control problems}.
\newblock Springer, New York, 1988.

\bibitem{McKel89}
R.~McKelvey.
\newblock Common property and the conservation of natural resources.
\newblock In S.A. Levin, T.G. Hallam, and L.J. Gross, editors, {\em Applied
  Mathematical Ecology}, pages 58--80. Springer, Berlin, 1989.

\bibitem{Pac98}
B.G. Pachpatte.
\newblock {\em Inequalities for differential and integral equations}.
\newblock Academic Press, San Diego, 1998.

\bibitem{RinZapSan12}
J.P. Rinc\'{o}n-Zapatero and M.S. Santos.
\newblock Differentiability of the value function in continuous-time economic
  models.
\newblock {\em J. Math. Anal. Appl.}, 394(1):305--323, 2012.

\bibitem{Roc70}
R.T. Rockafellar.
\newblock {\em Convex Analysis}.
\newblock Princeton University Press, Princeton, 1970.

\bibitem{RockWets09}
R.T. Rockafellar and R.J.-B. Wets.
\newblock {\em Variational analysis}.
\newblock Springer-Verlag, Berlin, 2009.

\bibitem{Rok00}
D.B. Rokhlin.
\newblock The derivative of the solution to the {B}ellman functional equation
  and the value of bioresources.
\newblock {\em Sib. Zh. Industr. Mat.}, 3(1(5)):169--181, 2000.

\bibitem{Son86}
H.M. Soner.
\newblock Optimal control with state-space constraint. {I}.
\newblock {\em SIAM J. Control Optim.}, 24(3):552--561, 1986.

\bibitem{Sri98}
S.M. Srivastava.
\newblock {\em A course on {B}orel sets}.
\newblock Springer-Verlag, New York, 1998.

\bibitem{Str96}
T.~Str{\"o}mberg.
\newblock The operation of infimal convolution.
\newblock {\em Diss. Math.}, 352:1--58, 1996.

\bibitem{Vill84}
A.~Villani.
\newblock On {L}usin's condition for the inverse function.
\newblock {\em Rend. Circ. Mat. Palermo}, 33(3):331--335, 1984.

\end{thebibliography}
\end{document}